\documentclass[11pt,reqno]{amsart}
\usepackage{indentfirst, amssymb,amsmath,amsthm}
\newtheorem{theorem}{Theorem}[section]

\newtheorem{lemma}{Lemma}[section]

\newtheorem{corollary}{Corollary}[section]
\newtheorem{definition}{Definition}[section]

\newcommand{\be}{\begin{equation}}
\newcommand{\ee}{\end{equation}}
\newcommand{\beas}{\begin{eqnarray*}}
\newcommand{\eeas}{\end{eqnarray*}}
\newcommand{\bea}{\begin{eqnarray}}
\newcommand{\eea}{\end{eqnarray}}

\numberwithin{equation}{section}

\begin{document}

\setlength{\unitlength}{1mm} \baselineskip .45cm
\setcounter{page}{1}
\pagenumbering{arabic}
\title[]
{almost co-K\"ahler manifolds and $(m,\rho)$-quasi-Einstein solitons}
\author[  ]
{ Krishnendu De, Mohammad Nazrul Islam Khan$^{*}$ and Uday Chand De   }

\address
 { Department of Mathematics,
 Kabi Sukanta Mahavidyalaya,
The University of Burdwan.
Bhadreswar, P.O.-Angus, Hooghly,
Pin 712221, West Bengal, India. ORCID iD: https://orcid.org/0000-0001-6520-4520}
\email{krishnendu.de@outlook.in }
\address
 {Department of Computer Engineering, College of Computer, Qassim University,
Buraydah 51452, Saudi Arabia.}
\email{m.nazrul@qu.edu.sa}
\address
{Department of Pure Mathematics, University of Calcutta, West Bengal, India. ORCID iD: https://orcid.org/0000-0002-8990-4609}
\email {uc$_{-}$de@yahoo.com}

\footnotetext {$\bf{2020\ Mathematics\ Subject\ Classification\:}.$ 53D15, 53C25.
\\ {Key words and phrases: Almost Co-K\"ahler manifold; $(\kappa,\mu)$-almost co-K\"ahler manifold; compact manifold; $(m,\rho)$-quasi-Einstein soliton.\\
\thanks{$^{*}$ Corresponding author}
}}
\maketitle
\begin{abstract}
The present paper aims to investigate $(m,\rho)$-quasi-Einstein metrices on almost co-K\"ahler manifolds $\mathcal{M}$. It is proven that if a $(\kappa,\mu)$-almost co-K\"ahler manifold with $\kappa<0$ is $(m,\rho)$-quasi-Einstein manifold, then $\mathcal{M}$ represents a $N(\kappa)$-almost co-K\"ahler manifold and the manifold is locally isomorphic to a solvable non-nilpotent Lie group. Next, we study the three dimensional case and get the above mentioned result along with the manifold $\mathcal{M}^3$ becoming an $\eta$-Einstein manifold. We also show that there does not exist $(m,\rho)$-quasi-Einstein structure on a compact $(\kappa,\mu)$-almost co-K\"ahler manifold of dimension greater than three with $\kappa<0$. Further, we prove that an almost co-K\"ahler manifold satisfying $\eta$-Einstein condition with constant coefficients reduces to a $K$-almost co-K\"ahler manifold, provided $ma_{1} \neq (2n-1)b_{1}$ and $m \neq 1$. We also characterize perfect fluid spacetime whose Lorentzian metric is equipped with $(m, \rho)$-quasi Einstein solitons and acquired that the perfect fluid spacetime has vanishing vorticity, or it represents dark energy era under certain restriction on the potential function. Finally, we construct an example of an almost co-K\"ahler manifold with $(m,\rho)$-quasi-Einstein solitons.
\end{abstract}	
	
\section{Introduction}
In \cite{bbw}, Boothby and Wang investigated an odd-dimensional differentiable manifold in 1958 with the help of almost contact and contact structure and investigated its features from a topological perspective. Making use of tensor calculus, Sasaki \cite{sas} described the characteristics of a differentiable manifold which is odd-dimensional, with contact structures in 1960. Such manifolds were referred to as contact manifolds. After that, other researchers have discovered various types of contact manifolds and investigated their characteristics. The Kenmotsu, the Sasakian, and the co-K\"ahler manifolds are Tanno's divisions of the almost contact metric manifolds \cite{tano}, whose automorphism groups have the highest dimensions. It should be noted that the cosymplectic manifolds that Blair \cite{5} introduced and Goldberg and Yano \cite{3} examined are nothing but the co-K\"ahler manifolds. The co-K\"ahler manifolds can be thought of from some topological perspectives as the analogue of K\"ahler manifolds in odd-dimension, that is why the new terminology has been introduced. The characteristics of almost co-K\"ahler manifolds, which are an extension of co-K\"ahler manifolds, were examined by several researchers. Perrone \cite{8} provided a thorough classification of the three-dimensional homogeneous almost co-K\"ahler manifolds as well as a local characterization of these manifolds under the assumption of local symmetry. We advise readers to read (\cite{chen}, \cite{18}, \cite{30}-\cite{21}) and the references therein for more information regarding almost co-K\"ahler manifolds.\par

In both mathematics and physics, Einstein manifolds are crucial. In Riemannian and Semi-Riemannian geometry, it is interesting to investigate Einstein manifolds and their generalizations. Numerous generalizations of Einstein manifolds have been developed recently, including quasi-Einstein manifolds \cite{25}, generalized quasi-Einstein manifolds \cite{12}, $m$-quasi-Einstein manifolds \cite{2}, $(m,\rho)$-quasi-Einstein manifolds \cite{ 13}, and many more.\par

A Riemannian metric $g$ of an almost co-K\"ahler manifold is named a Ricci soliton \cite{17} if there is a $\lambda_{1} \in \mathbb{R}=$ constant, $\mathbb{R}$ is the set of all real numbers and a smooth vector field $X$ such that
\be Ric+\frac{1}{2}\pounds_{X} g=\lambda_{1} g, \label{101}\ee
in which $Ric$ indicates the Ricci tensor of the metric tensor $g$ and $\pounds$ denotes the Lie-derivative. If $\lambda_{1}$ is a smooth function, then the above soliton is called an almost Ricci soliton \cite{dug}.\par

If we choose a smooth function $\omega:\mathcal{M}\to \mathbb{R}$ with $X=D\omega$, $D$ being the gradient operator of $g$, then it is named a gradient Ricci soliton. Hence, the equation (\ref{101}) reduces to
\be Ric+Hess\, \omega=\lambda_{1} g, \label{102} \ee
where $Hess$ is the Hessian operator.\par

A Ricci soliton is nothing but a natural generalization of Einstein metric \cite{11}. In this paper we study $(m,\rho)$-quasi-Einstein solitons which are the generalization of Einstein solitons and gradient Ricci solitons.\par

In an almost co-K\"ahler manifold the metric $g$, is named a generalized quasi-Einstein soliton if there exist $\omega$, $\alpha$ and $\beta$ (smooth functions) such that
\be Ric+Hess\, \omega-\alpha d\omega \otimes d\omega=\beta g. \ee
In particular, if $\beta \in \mathbb{R}$ and $\alpha=0$ , then the foregoing soliton reduces to a gradient Ricci soliton and $m$-quasi-Einstein soliton \cite{23}, if $\alpha=\frac{1}{m}$ and $\beta \in \mathbb{R}$, where $m \in \mathbb{N}$.
Catino \cite{12} presented the idea of a generalized quasi-Einstein soliton, and Huang and Wei \cite{13} further proposed taking into consideration the idea of a $(m,\rho)$-quasi-Einstein-soliton as its special case.

\begin{definition}
In an almost co-K\"ahler manifold the metric $g$ is named $(m,\rho)$-quasi-Einstein soliton if there exists a smooth function $\omega : \mathcal{M}^n \to \mathcal{R}$ and $m,\rho,\lambda_{1}\in \mathbb{R}=$ constant with $0<m \leq \infty$ such that
\be
Ric+Hess\,\omega-\frac{1}{m} d\omega \otimes d\omega=\beta g=(\rho \tau+\lambda_{1})g, \label{104}
\ee
where $\tau$ denotes the scalar curvature.

\end{definition}

Obviously, a $(\infty,0)$-quasi-Einstein soliton is a gradient Ricci soliton. In \cite{dede}, De and De investigated $(m,\rho)$-quasi-Einstein solitons in the frame-work of paracontact manifolds. Demirbe$\breve{{\rm g}}$ and G\"uler \cite{20} studied rigidity of $(m,\rho)$-quasi-Einstein manifolds.
As the study of almost co-K\"ahler manifolds with $(m,\rho)$-quasi-Einstein solitons is still pending, we want to fill this gap in this paper.  Precisely, we establish the subsequent theorems:

\begin{theorem}
If the metric of a $(\kappa,\mu)$-almost co-K\"ahler manifold $\mathcal{M}^{2n+1}$, $n>1$, with $\kappa<0$ is a $(m,\rho)$-quasi-Einstein metric, then $\mathcal{M}$ represents a $N(\kappa)$-almost co-K\"ahler manifold. Also, the manifold is locally isomorphic to a solvable non-nilpotent Lie-group $\mathcal{G}_\sigma$ with the almost co-K\"ahler structure $(\eta,\zeta,\varphi,g)$, where $\sigma=\sqrt{-\kappa}$.\label{001}
\end{theorem}

\begin{theorem}
Let the metric of a $(\kappa,\mu)$-almost co-K\"ahler manifold $\mathcal{M}^3$ with $\kappa<0$ is a $(m,\rho)$-quasi-Einstein metric. Then $\mathcal{M}$ represents a $N(\kappa)$-almost co-K\"ahler manifold. Also, $\mathcal{M}^3$ is locally isomorphic to $\mathcal{G}_\sigma$ with the almost co-K\"ahler structure, where $\sigma=\sqrt{-\kappa}$ and $\mathcal{M}^3$ becomes an $\eta$-Einstein manifold.\label{002}
\end{theorem}

\begin{theorem}
There does not exist $(m,\rho)$-quasi-Einstein structures $(g,\omega,\lambda_{1})$ with $D\omega=(\zeta \omega)\zeta$ on a compact $(\kappa,\mu)$-almost co-K\"ahler manifold with $\kappa<0$.\label{003}
\end{theorem}

\begin{theorem}
If an almost co-K\"ahler manifold $\mathcal{M}^{2n+1}$ satisfying $\eta$-Einstein condition $Ric=a_{1}g+b_{1}\eta \otimes \eta,\, b_{1} \neq 0$ with constant coefficients is $(m,\rho)$-quasi-Einstein manifold and $m \neq 1$, then $\mathcal{M}^{2n+1}$ is a $K$-almost co-K\"ahler manifold, provided $ma_{1} \neq (2n-1)b_{1}$.\label{004}
\end{theorem}
Solitons are actually waves that physically propagate with some energy loss and maintain their speed and shape after colliding with another wave of a similar kind. Solitons are crucial in the solution of initial-value problems in nonlinear partial differential equations describing wave propagation. Additionally, it describes the Fermi-Pasta-Ulam system's recurrence \cite{fpu}.\par
Interest in studying Ricci solitons and their generalizations in different geometrical contexts has also considerably increased due to their connection to general relativity. Recently, in perfect fluid spacetimes, many authors investigated many type of solitons like Ricci solitons \cite{dms}, gradient Ricci solitons (\cite{dez}, \cite{dms}), Yamabe solitons \cite{de}, gradient Yamabe solitons\cite{dez}, gradient m-quasi Einstein solitons\cite{dez}, gradient $\eta$-Einstein solitons(\cite{dms}), gradient Schouten solitons(\cite{dms}), respectively. Also, physical applications of almost Ricci solitons are investigated in (\cite{dug}, \cite{dug1}). Motivated by the above studies, in this article, we intend to investigate the $(m,\rho)$-quasi-Einstein solitons in perfect fluid spacetimes. As a result, we establish the following theorems.\par	
\begin{theorem}
\label{thm4.1}
 If a perfect fluid spacetime of dimension $n\geq 4$ admits a $(m,\rho)$-quasi Einstein soliton with the scalars $\alpha_{1}$, $\beta$ and the potential function $\omega$ remain invariant under the velocity vector field $\rho$, then either the perfect fluid spacetime has vanishing vorticity or, the spacetime represents dark energy era.
\end{theorem}	
	
\section{  almost co-K\"ahler manifolds}
$\mathcal{M}^{2n+1}$, a smooth manifold is named an almost contact metric manifold if there exist a vector field $\zeta$, a tensor field $\varphi$, a one-form $\eta$ and the Riemannian metric $g$ satisfying
\be
 \varphi^2+I=\eta \otimes \zeta, \quad \eta(\zeta)=1, \label{201}
\ee
	
\be
 g(E_{1},F_{1})=g(\varphi E_{1},\varphi F_{1})+\eta(E_{1})\eta(F_{1}) \label{202}
 \ee
for all vector fields $E_{1},F_{1}$ on $\mathcal{M}^{2n+1}$. From (\ref{201}), we get $\eta \circ \varphi=0$ and $\varphi \zeta=0$. \par

To each $\mathcal{M}^{2n+1}$ there is an associated 2-form $\Phi(E_{1},F_{1})=g(E_{1},\varphi F_{1})$ for all vector fields $E_{1},F_{1}$ on $\mathcal{M}^{2n+1}$. $\eta \wedge \Phi^n$ is non-vanishing everywhere and hence $\mathcal{M}^{2n+1}$ is orientable. \par
The Riemannian curvature tensor $K$ is described by
	\be K(E_{1},F_{1})=[\nabla_{E_{1}},\nabla_{F_{1}}]-\nabla_{[E_{1},F_{1}]}, \label{203} \ee
in which $\nabla$ denotes the Riemannian connection.\par

The manifold $\mathcal{M}$ is named an almost co-K\"ahler manifold \cite{3} if  $d\eta=0$ and $d\Phi=0$. In particular, an almost co-K\"ahler manifold is called co-K\"ahler if it is normal, that is,  $\nabla \Phi=0$. An almost co-K\"ahler structure is said to be strictly almost co-K\"ahler if it is not a co-K\"ahler structure.\par

On an almost co-K\"ahler manifold $\mathcal{M}^{2n+1}(\eta,\zeta,\varphi,g)$, we set $2h=\pounds_{\zeta} \varphi$ and $h'=h \circ \varphi$. It is known that $h$ and $h'$ are symmetric $(1,1)$-tensor fields and satisfy the following relations (\cite{6}, \cite{7}, \cite{8})
	\be h\zeta=0,\quad \varphi h+h\varphi=0,\quad {\rm tr}\,h=0={\rm tr}\,h', \label{204} \ee
	\be \nabla_\zeta \varphi=0,\quad \nabla \zeta=h',\label{205} \ee
	\be \nabla_\zeta h=-h^2 \varphi-\varphi \ell, \label{206} \ee
	\be \varphi \ell \varphi-\ell=2h^2, \label{207} \ee
	\be Ric(\zeta,\zeta)+{\rm tr}\, h^2=0, \label{208} \ee
where $\ell=K(.,\zeta)\zeta$ indicates the Jacobi operator and ${\rm tr}$ denotes the trace.

\begin{definition}
An almost co-K\"ahler manifold is named a $K$-almost co-K\"ahler manifold if the characteristic vector field $\zeta$ is Killing.
\end{definition}

It is known that, any co-K\"ahler manifold is a $K$-almost co-K\"ahler manifold, but the converse is not true, in general. But it is true in 3-dimensinal manifold.
\begin{lemma}
		\cite{3} Any 3-dimensional almost co-K\"ahler manifold is co-K\"ahler if and only if $\zeta$ is Killing.\label{006}
	\end{lemma}
	
\section{ $(\kappa,\mu)$- almost co-K\"ahler manifolds}
	
Let $\mathcal{M}^{2n+1}$ be an almost co-K\"ahler manifold. Then $\mathcal{M}^{2n+1}$ is named a $(\kappa,\mu)$-almost co-K\"ahler manifold if $\xi$ belongs to the $(\kappa,\mu)$-nullity distribution, that is,
\be K(E_{1},F_{1})\xi=\kappa(\eta(F_{1})E_{1}-\eta(E_{1})F_{1})+\mu(\eta(F_{1})hE_{1}-\eta(E_{1})hF_{1}) \label{301} \ee
for all vector fields $E_{1},F_{1}$ on $\mathcal{M}^{2n+1}$ and $(\kappa,\mu)\in \mathcal{R}^2$. Such a manifold was first presented by Endo \cite{1} and were generalized by Dacko and Olszak \cite{22} to $(\kappa,\mu,\nu)$-spaces. Replacing $F_{1}$ by $\xi$ in (\ref{301}), we acquire $\ell=-\kappa \varphi^2+\mu h$. Putting the value of $\ell$ in (\ref{207}) and using second equation of (\ref{204}), we obtain
\be
 h^2=\kappa \varphi^2. \label{302}
\ee
By (\ref{302}), we easily see that $\kappa=0$ and $\kappa \leq 0$  if and only if $\mathcal{M}^{2n+1}$ is a $K$-almost co-K\"ahler manifold. If $\kappa<0$, then $\mathcal{M}^{2n+1}$ is a strictly almost co-K\"ahler manifold. In particular, if $\mu=0$, then $\xi$ belongs to the $\kappa$-nullity distribution $N(\kappa)$ described by
\beas &&N(\kappa) : p \to N_p(\kappa)\\&=&\!\!\! \{G_{1} \in T_p(\mathcal{M}^{2n+1}) : K(E_{1},F_{1})G_{1}\\&&
=\kappa(g(F_{1},G_{1})E_{1}-g(E_{1},G_{1})F_{1}),\, \forall \, E_{1},F_{1} \in T_p(\mathcal{M}^{2n+1})\},
\eeas
where $T_p(\mathcal{M}^{2n+1})$ is the tangent space of $\mathcal{M}^{2n+1}$ at $p\in \mathcal{M}^{2n+1}$ and investigated by Dacko \cite{14}. According to Theorem 4 of \cite{14}, we have

\begin{theorem}
Let $\mathcal{M}^{2n+1}$ be an almost co-K\"ahler manifold with $\kappa<0$. Then $\mathcal{M}^{2n+1}$ is locally isomorphic to a solvable non-nilpotent Lie group $\mathcal{G}_\sigma$ with the almost co-K\"ahler structure $(\eta,\xi,\varphi,g)$, where $\sigma=\sqrt{-\kappa}$ and underlying manifold of $\mathcal{G}_{\sigma}$ is the space $\mathbb{R}^{2n+1}$.\label{007}
\end{theorem}

\begin{lemma} \cite{4}
Let $\mathcal{M}^{2n+1}$ be a $(\kappa,\mu)$-almost co-K\"ahler manifold with $\kappa<0$. Then the Ricci operator $Q$ is described by
\be Q=2n\kappa \eta \otimes \xi+\mu h. \label{303} \ee
\end{lemma}

Thought the paper we assume $(\kappa,\mu)$-almost co-K\"ahler manifolds in the sense of Endo \cite{1}, that is $\kappa$ and $\mu$ are constant.

\section{ Proof of the results}
	Naturally, there arises a question regarding the existence of Theorem \ref{007} for $(\kappa,\mu)$-almost co-K\"ahler manifolds with $(m,\rho)$-quasi-Einstein solitons. The affirmative answer of the above question is given by Theorem \ref{001}.\par

	{\bf Proof of the Theorem \ref{001}.}
From (\ref{303}), we infer that $\tau=2n\kappa={\rm constant}$.\\ By virtue of (\ref{104}) and (\ref{303}), we obtain
\be \nabla_{E_{1}} D\omega=(2n\kappa \rho+\lambda_{1})E_{1}-\mu hE_{1}-2n\kappa \eta(E_{1})\zeta+\frac{1}{m}g(E_{1},D\omega)D\omega \label{304} \ee
for all vector field $E_{1}$ on $\mathcal{M}^{2n+1}$. Taking covariant derivative of (\ref{304}) along $F_{1}$ and using (\ref{205}), we get
\bea
\nabla_{F_{1}}\nabla_{E_{1}} D\omega&=&(2n\kappa \rho+\lambda_{1})\nabla_{F_{1}} E_{1}-\mu\nabla_{F_{1}}(hE_{1})\nonumber\\&-&2n\kappa (F_{1}\eta(E_{1}))\zeta-2n\kappa \eta(E_{1})h'F_{1}\nonumber\\&+&\frac{1}{m}[g(\nabla_{F_{1}}E_{1},D\omega)D\omega+g(E_{1},\nabla_{F_{1}} D\omega)D\omega \nonumber\\&+&g(E_{1},D\omega)\nabla_{F_{1}} D\omega]. \label{305}
\eea
Interchanging $E_{1}$ and $F_{1}$ in the above equation, we have
\bea
\nabla_{E_{1}}\nabla_{F_{1}} D\omega&=&(2n\kappa \rho+\lambda_{1})\nabla_{E_{1}} F_{1}-\mu\nabla_{E_{1}}(hF_{1})\nonumber\\&-&2n\kappa (E_{1}\eta(F_{1}))\zeta-2n\kappa \eta(F_{1})h'E_{1}\nonumber\\&+&\frac{1}{m}[g(\nabla_{E_{1}}F_{1},D\omega)D\omega+g(F_{1},\nabla_{E_{1}} D\omega)D\omega \nonumber\\&+&g(F_{1},D\omega)\nabla_{E_{1}} D\omega]. \label{306}
\eea

Using (\ref{304}), (\ref{305}) and (\ref{306}) in (\ref{203}), we infer that
\bea
K(E_{1},F_{1})D\omega&=&-\mu(\nabla_{E_{1}} h)F_{1}+\mu (\nabla_{F_{1}} h)E_{1}\nonumber\\&-&2n\kappa(\eta(F_{1})h'E_{1}-\eta(E_{1})h'F_{1})\nonumber\\&+&\frac{2n\kappa\rho+\lambda_{1}}{m}((F_{1}\omega)E_{1}-(E_{1}\omega)F_{1})\nonumber\\&-&\frac{\mu}{m}((F_{1}\omega)hE_{1}-(E_{1}\omega)hF_{1})\nonumber\\&-&\frac{2n\kappa}{m}((F_{1}\omega)\eta(E_{1})-(E_{1}\omega)\eta(F_{1}))\zeta. \label{307}
 \eea
Taking inner product of the previous equation with $\zeta$ and utilizing (\ref{302}), we obtain
	\bea g(K(E_{1},F_{1})D\omega,\zeta)=2\mu \kappa g(E_{1},\varphi F_{1})+\frac{2n\kappa(\rho-1)+\lambda_{1}}{m}((F_{1}\omega)\eta(E_{1})-(E_{1}\omega)\eta(F_{1})). \label{308} \eea
From (\ref{301}) and (\ref{308}), it follows that
	\bea &&\kappa((F_{1}\omega)\eta(E_{1})-(E_{1}\omega)\eta(F_{1}))+\mu(g(hF_{1},D\omega)\eta(E_{1})-g(hE_{1},D\omega)\eta(F_{1})) \nonumber\\&=&2\mu \kappa g(E_{1},\varphi F_{1})+\frac{2n\kappa(\rho-1)+\lambda_{1}}{m}((F_{1}\omega)\eta(E_{1})-(E_{1}\omega)\eta(F_{1})). \label{309} \eea
	Replacing $E_{1}$ by $\varphi E_{1}$ and $F_{1}$ by $\varphi F_{1}$ in the foregoing equation, we get
	$$0=2\mu\kappa g(E_{1},\varphi F_{1})$$
	which implies that $\mu=0$, because $\kappa<0$.\par
If $\mu=0$, the $(\kappa,\mu)$-almost co-K\"ahler manifold reduces to a $N(\kappa)$-almost co-K\"ahler manifold.\par

Here $\zeta$ belongs to $\kappa$-nullity distribution with $\kappa<0$. By Theorem \ref{007} $\mathcal{M}^{2n+1}$ is locally isomorphic to a solvable non-nilpotent Lie group $\mathcal{G}_\sigma$, where $\sigma=\sqrt{-\kappa}$.\par

Thus the proof is completed.\par

It is well known that on any 3-dimensional manifold, the curvature tensor is described by
\bea
K(E_{1},F_{1})G_{1}&=&Ric(F_{1},G_{1})E_{1}-Ric(E_{1},G_{1})F_{1}+g(F_{1},G_{1})QE_{1}-g(E_{1},G_{1})QF_{1}\nonumber\\&-&\frac{\tau}{2}(g(F_{1},G_{1})E_{1}-g(E_{1},G_{1})F_{1}) \label{313}
\eea
for all vector fields $E_{1},F_{1},G_{1}$ on $\mathcal{M}^3$.\par
 Taking trace of (\ref{301}), we get $Q\zeta=2\kappa \zeta$. Putting $F_{1}=G_{1}=\zeta$ in (\ref{313}) and using $\ell=-\kappa \varphi^2+\mu h$, we obtain
	\be Q=\left(\frac{\tau}{2}-\kappa\right)I+\left(3\kappa-\frac{\tau}{2}\right)\eta \otimes \zeta+\mu h. \label{314} \ee
This is the expression of the Ricci operator $Q$ on any 3-dimensional $(\kappa,\mu)$-almost co-K\"ahler manifold.\par

	{\bf Proof of the Theorem \ref{002}.} Using (\ref{314}) in (\ref{104}), we get
\beas
\nabla_{E_{1}} D\omega&=&\left(\rho \tau+\lambda_{1}-\frac{\tau}{2}+\kappa\right)E_{1}+\left(\frac{\tau}{2}-3\kappa\right)\eta(E_{1})\zeta\nonumber\\&-&\mu hE_{1}+ \frac{1}{m}g(E_{1},D\omega)D\omega
\eeas
for all $E_{1}$ on $\mathcal{M}^3$. By direct computation, we have
\bea
 K(E_{1},F_{1})D\omega&=&\left(\rho -\frac{1}{2}\right)(E_{1}\tau)F_{1}-\left(\rho-\frac{1}{2}\right)(F_{1}\tau)E_{1}\nonumber\\&+&\frac{1}{2}(E_{1}\tau)\eta(F_{1})\zeta-\frac{1}{2}(F_{1}\tau)\eta(E_{1})\zeta\nonumber\\&+&\left(\frac{\tau}{2}-3\kappa\right)(\eta(F_{1})h'E_{1}-\eta(E_{1})h'F_{1})\nonumber\\&-&-\mu (\nabla_{E_{1}} h)F_{1}+\mu (\nabla_{F_{1}} h)E_{1}\nonumber\\&+&\frac{1}{m}\left(\rho \tau+\lambda_{1}-\frac{\tau}{2}+\kappa\right)((F_{1}\omega)E_{1}-(E_{1}\omega)F_{1})\nonumber\\&+&\frac{1}{m}\left(\frac{\tau}{2}-3\kappa\right)((F_{1}\omega)\eta(E_{1})-(E_{1}\omega)\eta(F_{1}))\nonumber\\&-&\frac{\mu}{m}((F_{1}\omega)h(E_{1})-(E_{1}\omega)h(F_{1})) \label{315}
\eea
for all $E_{1},F_{1}$ on $\mathcal{M}^3$. Taking inner product of (\ref{315}) with $\zeta$, we obtain
\bea
g(K(E_{1},F_{1})Df,\zeta)&=&\rho (E_{1}\tau)\eta(F_{1})-\rho (F_{1}\tau)\eta(E_{1})+ 2\mu \kappa g(E_{1},\varphi F_{1})\nonumber\\&+&\frac{1}{m}(\rho \tau+\lambda_{1}-2\kappa)((F_{1}\omega)\eta(E_{1})-(E_{1}\omega)\eta(F_{1})), \label{316}
\eea
where we have used (\ref{302}). By virtue of (\ref{301}) and (\ref{316}), we have
\bea  &&\kappa((F_{1}\omega)\eta(E_{1})-(E_{1}\omega)\eta(F_{1}))+\mu(g(hF_{1},Df)\eta(E_{1})-g(hE_{1},Df)\eta(F_{1}))\nonumber\\&=&\rho (E_{1}\tau)\eta(F_{1})-\rho (F_{1}\tau)\eta(E_{1})+ 2\mu \kappa g(E_{1},\varphi F_{1})\nonumber\\&+&\frac{1}{m}(\rho \tau+\lambda_{1}-2\kappa)((F_{1}\omega)\eta(E_{1})-(E_{1}\omega)\eta(F_{1})). \label{317}
\eea
Now, replacing $E_{1}$ and $F_{1}$ by $\varphi E_{1}$ and $\varphi F_{1}$, respectively, in (\ref{317}) we get
	$$0=2\mu \kappa g(E_{1},\varphi F_{1}),$$
which implies that $\mu=0$, because $\kappa<0$.\par

If $\mu=0$, the $(\kappa,\mu)$-almost co-K\"ahler manifold turns into a $N(\kappa)$-almost co-K\"ahler manifold.\par

By Theorem \ref{007} $\mathcal{M}^3$ is locally isometric to a solvable non-nilpotent Lie group $\mathcal{G}_\sigma$, where $\sigma=\sqrt{-\kappa}$. Since $\mu=0$, from (\ref{314}) it follows that $\mathcal{M}^3$ is an $\eta$-Einstein manifold.\par

Hence the proof.\par

{\bf Proof of the Theorem \ref{003}.} From (\ref{303}), we acquire that the scalar curvature $\tau=2n\kappa$.
	Differentiating $D\omega=(\zeta \omega)\zeta$ along the arbitrary vector field $E_{1}$, we get
	\be \nabla_{E_{1}} D\omega=(E_{1}(\zeta \omega))\zeta+(\zeta \omega)h'E_{1}. \label{319} \ee
	By (\ref{303}) and (\ref{319}), equation (\ref{104}) takes the form
	\bea &&2n\kappa \eta(E_{1})\eta(F_{1})+\mu g(hE_{1},F_{1})+(E_{1}(\zeta \omega))\eta(F_{1})+(\zeta \omega)g(h'E_{1},F_{1})\nonumber\\&-&\frac{1}{m}(\zeta \omega)\eta(E_{1})\eta(F_{1})=(2n\kappa\rho+\lambda_{1})g(E_{1},F_{1}) \label{320} \eea
	for all vector fields $E_{1},F_{1}$ on $\mathcal{M}^{2n+1}$. Replacing $E_{1}$ by $\varphi E_{1}$ and $F_{1}$ by $\varphi F_{1}$, we obtain
	$$ -\mu g(hE_{1},F_{1})-(\zeta f)g(h'E_{1},F_{1})=(2n\kappa\rho+\lambda_{1})g(\varphi E_{1},\varphi F_{1}). $$
	Contracting the preceding equation, we get
	\be 2n\kappa\rho+\lambda_{1}=0.\label{800}\ee
	Setting $E_{1}=F_{1}=\zeta$ in (\ref{320}) and using (\ref{800}), we have
	\be 2n\kappa+\zeta(\zeta \omega)-\frac{1}{m}(\zeta \omega)^2=0. \label{321} \ee
	Contracting in (\ref{319}), we find
	\be \Delta \omega+\zeta(\zeta \omega)=0, \label{322} \ee
	where $\Delta=-{\rm div}\,D$ is the Laplacian operator.
	Using (\ref{322}) in (\ref{321}), we obtain
	$$\Delta \omega=2n\kappa-\frac{1}{m}(\zeta \omega)^2.$$
	By divergence theorem
	$$\frac{1}{m}\int_{\mathcal{M}} (\zeta \omega)^2 d\mathcal{M}=2n\kappa \int_{\mathcal{M}} d\mathcal{M}.$$
$d\mathcal{M}$ represents the volume form of the manifold and is positive, since $\mathcal{M}$ is orientable. Therefore, the right hand side is negative, because $\kappa <0$. Thus, the above relation does not hold.\par

 This completes the proof.\par

{\bf $\eta$-Einstein  almost co-K\"ahler manifolds:}

An almost co-K\"ahler manifold $(\mathcal{M}^{2n+1},g)$ is named an $\eta$-Einstein manifold if
\be
 Ric=a_{1}g+b_{1}\eta \otimes \eta, \label{401}
\ee
where $a_{1}$ and $b_{1}$ are smooth functions. For $b_{1}=0$, it reduces to an Einstein manifold. In particular, if both $a_{1}$ and $b_{1}$ are constant, then we say that $\mathcal{M}^{2n+1}$ satisfies $\eta$-Einstein condition with constant coefficients. From (\ref{208}), we have
\be
a_{1}+b_{1}=-{\rm tr}\, h^2. \label{402}
\ee
The $\eta$-Einstein condition on an almost co-K\"ahler manifold implies that $\zeta$ is a harmonic vector field \cite{9}.\par

{\bf Proof of the Theorem \ref{004}.} From (\ref{401}), we infer that scalar curvature $\tau=(2n+1)a_{1}+b_{1}={\rm constant}$.
By virtue of (\ref{401}), we obtain from (\ref{104})
\be
\nabla_{E_{1}} D\omega=(\rho \tau+\lambda_{1}-a_{1})E_{1}-b_{1}\eta(E_{1})\zeta+\frac{1}{m}g(E_{1},D\omega)D\omega \label{403}
\ee
for all vector field  $E_{1}$ on $\mathcal{M}^{2n+1}$. Using this we compute
\bea
 K(E_{1},F_{1})D\omega&=&-b_{1}\eta(F_{1})h'E_{1}+b_{1}\eta(E_{1})h'F_{1}\nonumber\\&+&\frac{1}{m}(\rho \tau+\lambda_{1}-a_{1})((F_{1}\omega)E_{1}-(E_{1}\omega)F_{1})\nonumber\\&-&\frac{b_{1}}{m}((F_{1}\omega)\eta(E_{1})-(E_{1}\omega)\eta(F_{1}))\zeta. \label{404}
 \eea
Contracting (\ref{404}), we get
\be
 Ric(F_{1},D\omega)=\frac{2n(\rho \tau+\lambda_{1}-a_{1})}{m}(F_{1}\omega)-\frac{b_{1}}{m}((F_{1}\omega)-(\zeta \omega)\eta(F_{1})), \label{405}
\ee
where we have used ${\rm tr}\, h'=0$.
From (\ref{401}) and (\ref{405}), it can be written as
\be
 a_{1}(F_{1}\omega)+b_{1}(\zeta \omega)\eta(F_{1})=\frac{2n(\rho \tau+\lambda_{1}-a_{1})}{m}(F_{1}\omega)-\frac{b_{1}}{m}((F_{1}\omega)-(\zeta \omega)\eta(F_{1})). \label{406}
\ee
Replacing $F_{1}$ by $\zeta$ in (\ref{406}), we have
	$$\left(a_{1}+b_{1}-\frac{2n(\rho \tau+\lambda_{1}-a_{1})}{m}\right)(\zeta \omega)=0.$$
From the preceding equation, we infer either $a_{1}+b_{1}-\displaystyle\frac{2n(\rho \tau+\lambda_{1}-a_{1})}{m}=0$, or $\zeta \omega=0$.\par

If $a_{1}+b_{1}-\displaystyle\frac{2n(\rho \tau+\lambda_{1}-a_{1})}{m}=0$, from (\ref{406}) we can find
$$ \frac{b_{1}(1-m)}{m}((F_{1}\omega)-(\zeta \omega)\eta(F_{1}))=0, $$
which implies $D\omega=(\zeta \omega)\zeta$, since $b_{1}\neq 0$ and $m\neq 1$. Putting $D\omega=(\zeta \omega)\zeta$ in (\ref{403}) and using (\ref{205}), we get
\be
(E_{1}(\zeta \omega))\zeta+(\zeta \omega)h'E_{1}=\frac{m(a_{1}+b_{1})}{2n}E_{1}-b_{1}\eta(E_{1})\zeta+\frac{1}{m}(\zeta \omega)^2\eta(E_{1})\zeta. \label{407}
\ee
Taking trace of (\ref{407}), we have
\be
\zeta(\zeta \omega)=\frac{m(a_{1}+b_{1})}{2n}(2n+1)-b_{1}+\frac{1}{m}(\zeta \omega)^2.
\ee
Setting $E_{1}=\zeta$ in (\ref{407}) and after that taking inner product, we acquire
\be
 \zeta(\zeta \omega)=\frac{m(a_{1}+b_{1})}{2n}-b_{1}+\frac{1}{m}(\zeta \omega)^2.
 \ee
The foregoing two equations imply that $a_{1}+b_{1}=0$. From (\ref{402}), ${\rm tr}\,h^2=0$. Since $h$ is a symmetric tensor, $h=0$. Hence $\zeta$ is a Killing vector field. Consequently, $\mathcal{M}^{2n+1}$ is a $K$-almost co-K\"ahler manifold.\par

On the other hand if $\zeta \omega=0$, then from (\ref{406}) we get
	\be \left(a_{1}-\frac{2n(\rho \tau+\lambda_{1}-a_{1})}{m}+\frac{b_{1}}{m}\right)(F_{1}\omega)=0. \label{410} \ee
If $F_{1}\omega=0$, then $\omega$ is a constant and from (\ref{104}) we see that the manifold is Einstein, which is a contradiction because $b_{1}\neq 0$. Therefore, from (\ref{410}) we observed that
	\be a_{1}-\frac{2n(\rho \tau+\lambda_{1}-a_{1})}{m}+\frac{b_{1}}{m}=0. \label{411} \ee
	Putting $E_{1}=\zeta$ in (\ref{403}) and then taking inner product with $\zeta$, we get
	\be \rho \tau+\lambda_{1}-a_{1}-b_{1}=0, \label{412} \ee
where we have used $\nabla_\zeta \zeta=0$. From (\ref{411}) and (\ref{412}), it follows that $ma_{1}=(2n-1)b_{1}$.\par
Thus the proof is finished.\par

Using the Lemma \ref{006}, we can state that
\begin{corollary}
If an almost co-K\"ahler manifold $\mathcal{M}^3$ satisfies $\eta$-Einstein condition $Ric=a_{1}g+b_{1}\eta \otimes \eta$, $b_{1}\neq 0$ with constant coefficients is a $(m,\rho)$-quasi-Einstein manifold and $m \neq 1$, then $\mathcal{M}^3$ is a co-K\"ahler manifold, provided $ma_{1} \neq b_{1}$.
\end{corollary}
By Lemma 5 of \cite{14}, if $\mathcal{M}^{2n+1}$ is an almost co-K\"ahler manifold whose structure vector field $\zeta$ belongs to the $\kappa$-nullity distribution with $\kappa<0$, then $\mathcal{M}$ satisfies $\eta$-Einstein condition with constant coefficients with $a_{1}=0$ and $b_{1}=2n\kappa \neq 0$. Also $ma_{1} \neq (2n-1)b_{1}$. Thus, from Theorem \ref{004}, we can state that
\begin{corollary}
There does not exist $(m,\rho)$-quasi-Einstein structure with $m \neq 1$ on an almost co-K\"ahler manifold whose structure vector field belongs to $\kappa$-nullity distribution with $\kappa<0$.
\end{corollary}

\section{ Applications to physics}
A Lorentzian manifold $\mathcal{N}^{n}$ is called a perfect fluid spacetime if its non-vanishing Ricci tensor $Ric$ obeys
\begin{equation}
\label{1}
Ric=\alpha_{1} g+ \beta_{1} A \otimes A,
\end{equation}
where $\alpha_{1}$, $\beta_{1}$ are scalar (not simultaneously zero) and for all $E_{1}$, $g(E_{1}, \rho)=A(E_{1})$, $\rho$ being the velocity vector field. We know that in a perfect fluid spacetime, the vector field $\rho$ is unit timelike, hence, we have $g(\rho, \rho)=-1$. \par

From (\ref{1}), we get
\begin{equation}
\label{3}
Q E_{1}=\alpha_{1} E_{1}+\beta_{1} A(E_{1}) \rho,
\end{equation}
where $Q$ is the Ricci operator defined by $g(QE_{1}, F_{1})=Ric(E_{1},F_{1})$
and contracting $E_{1}$ in (\ref{3}), we infer
\begin{equation}
\label{2}
r=\sum_{i} \epsilon_{i}Q E_{i}=n \alpha_{1}-\beta_{1},
\end{equation}

where $\{E_{i}\}$ is an orthonormal basis of the tangent space at each point of the spacetime and $\epsilon_{i}=g(E_i, E_i) =\pm1$.
The covariant derivative of (\ref{3}) gives
\begin{equation}
\label{5b}
(\nabla_{E_{1}}Q)(F_{1})=E_{1} (\alpha_{1})F_{1}+E_{1}(\beta_{1}) A(F_{1})\rho+\beta_{1} (\nabla_{E_{1}}A)(F_{1})\rho+\beta_{1} A(F_{1}) \nabla_{E_{1}}\rho.
\end{equation}

If $k$ denotes the gravitational constant and $T$ indicates the energy momentum tensor, then Einstein's field equations in the absence of cosmological constant have the structure
\begin{equation}
\label{1.2}
Ric-\frac{r}{2}g=k T.
\end{equation}
In case of perfect fluid spacetime if $p$ indicates isotropic pressure and $\sigma$ is the energy density, then $T$ is defined as
\begin{equation}
\label{1.1}
T=(p+\sigma)A \otimes A+p g.
\end{equation}

The necessary and sufficient condition for the constant scalar curvature of a perfect fluid spacetime is that $n E_{1}(\alpha_{1})=E_{1}(\beta_{1}).$
Combining the equations (\ref{1}), (\ref{1.2}) and (\ref{1.1}), we infer that
\begin{equation}
\label{1.7}
\alpha_{1}=\frac{k(p-\sigma)}{2-n},\;\beta_{1}=\kappa(p+\sigma).
\end{equation}

Furthermore, characterizing the particular sort of perfect fluid spacetimes, $\sigma$ and $p$ are interconnected by the relation $p = \Omega \sigma $ in which $\Omega$ is named the equation of state (EOS) parameter. If the state equation is of the form $p = p(\sigma )$, then the perfect fluid spacetime is called isentropic. Also, if $p = \sigma$, then the perfect fluid spacetime represents stiff matter fluid\cite{ch1}. The equation of state of stiff matter fluid was introduced by Zeldovich \cite{ze} to describe a cold gas of baryons, and used in his cosmological model.\par
A dust solution, in general relativity, is one in which the gravitational field is totally produced by the mass, momentum, and stress density of a perfect fluid with positive mass density but vanishing pressure. This is a form of exact solution to the Einstein field equation. If $p = 0$, perfect fluid spacetime represents the dust matter fluid whereas the perfect fluid spacetime represents the radiation era with $p -\frac{\sigma}{3}=0$ and the dark energy era with $p=-\sigma $ \cite{ch1}.\par
Moreover, Dark energy era that satisfies the equation of state with $\Omega < -1$ is known as phantom regime or phantom energy. It has positive energy density but negative pressure, such that $p+\sigma < 0$.
The physical consequences are explored in (\cite{cal}, \cite{cal1}). \par

{\bf Proof of the Theorem \ref{thm4.1} .}

Let the perfect fluid spacetime of dimension $n\geq 4$ admit a $(m,\rho)$-quasi Einstein metric and at first, we prove the following result
\begin{lemma}\label{lem1}
Every perfect fluid spacetime of dimension $n\geq 4$ obeying $(m,\rho)$-quasi Einstein soliton satisfies the following:
\begin{eqnarray}\label{k2}
&&K(E_{1},F_{1})D \omega = (\nabla_{F_{1}}Q)E_{1}-(\nabla_{E_{1}} Q)F_{1}+\frac{\beta}{m}\{ F_{1}( \omega)E_{1}-E_{1} (\omega)F_{1}\}\nonumber\\&&
+\frac{1}{m}\{ E_{1} (\omega)Q F_{1}-F_{1} (\omega)Q E_{1} \}+\{(E_{1}\beta)F_{1}-(F_{1}\beta)E_{1}\},
\end{eqnarray}
for all $E_{1}, \, F_{1}$.
\end{lemma}
\begin{proof}
By hypothesis, the perfect fluid spacetime is endowed with $(m,\rho)$-quasi Einstein metric. Therefore, using (\ref{104}) we can write
\begin{equation}\label{k3}
\nabla_{E_{1}}D \omega+Q E_{1}=\frac{1}{m}g(E_{1},D \omega)D \omega+\beta E_{1}.
\end{equation}
Differentiating the equation (\ref{k3}) covariantly along $F_{1}$, we get
\begin{eqnarray}\label{k4}
\nabla_{F_{1}}\nabla_{E_{1}}D \omega &=& -\nabla_{F_{1}}Q E_{1}+ \frac{1}{m} \nabla_{F_{1}}g(E_{1},D \omega)D \omega\nonumber\\&& +\frac{1}{m} g(E_{1},D \omega)\nabla_{F_{1}}D \omega+ \beta \nabla_{F_{1}}E_{1}+(F_{1}\beta)E_{1}.
\end{eqnarray}
Exchanging $E_{1}$ and $F_{1}$ in (\ref{k4}), we lead
\begin{eqnarray}\label{k5}
\nabla_{E_{1}}\nabla_{F_{1}}D \omega &=& -\nabla_{E_{1}}Q  F_{1}+ \frac{1}{m}\nabla_{E_{1}}g(F_{1},D \omega)D \omega\nonumber\\&& +\frac{1}{m} g(F_{1},D \omega)\nabla_{E_{1}}D \omega +\beta \nabla_{E_{1}}F_{1}+(E_{1}\beta)F_{1}
\end{eqnarray}
and
\begin{equation} \label{k6}
\nabla_{[E_{1},F_{1}]}D \omega = -Q[E_{1},F_{1}]+ \frac{1}{m}g([E_{1},F_{1}],D \omega)D \omega+\beta [E_{1},F_{1}].
\end{equation}
Using (\ref{k3})-(\ref{k6}) and the relation $K(E_{1},F_{1})D \omega  =\nabla_{E_{1}}\nabla_{F_{1}}D \omega-\nabla_{F_{1}}\nabla_{E_{1}}D \omega-\nabla_{[E_{1},F_{1}]}D \omega$, we acquire
\begin{eqnarray}
K(E_{1},F_{1})D \omega  &=& (\nabla_{F_{1}}Q)E_{1}-(\nabla_{E_{1}}Q)F_{1}  +\frac{\beta}{m}\{ F_{1} (\omega)E_{1}-E_{1} (\omega)F_{1}\}\nonumber\\&&
+\frac{1}{m} \{ E_{1} (\omega)Q F_{1}-F_{1} (\omega)Q E_{1} \}+\{(E_{1}\beta)F_{1}-(F_{1}\beta)E_{1}\}.\nonumber
\end{eqnarray}
\end{proof}
Utilizing the previous Lemma and the equations (\ref{3}), (\ref{5b}), we infer
\begin{eqnarray}
\label{kk6}
&&K(E_{1}, F_{1})D\omega=E_{1} (\alpha_{1})F_{1}-F_{1} (\alpha_{1})E_{1}+\{E_{1} (\beta_{1})A(F_{1})-F_{1} (\beta_{1})A(E_{1})\}\rho\nonumber\\&&
+ \beta_{1} \{(\nabla_{E_{1}} A)(F_{1})\rho+ A(F_{1}) \nabla_{E_{1}} \rho-\nabla_{F_{1}} A)(E_{1})\rho- A(E_{1}) \nabla_{F_{1}} \rho \}\nonumber\\&&
+\frac{\beta}{m}\{F_{1} (\omega)E_{1}-E_{1} (\omega)F_{1}\}+\frac{1}{m}\{\alpha_{1} E_{1} (\omega)F_{1}
+\beta_{1} E_{1} (\omega)A(F_{1})\rho\nonumber\\&&
-\alpha_{1} F_{1} (\omega)E_{1}-\beta_{1} F_{1} (\omega)A(E_{1})\rho\}+\{(E_{1}\beta)F_{1}-(F_{1}\beta)E_{1}\}.
\end{eqnarray}
Taking a set of orthonormal frame field and executing contraction of the equation (\ref{kk6}), we get
\begin{eqnarray}
\label{k8}
Ric(E_{1}, D\omega)&=&(1-n)E_{1} (\alpha_{1})+E_{1} (\beta_{1})+\rho (\beta_{1})A(E_{1})\nonumber\\&&
+\beta_{1} \{(\nabla_{E_{1}} A)(F_{1})-\nabla_{E_{1}} A)(\rho)+ A(E_{1}) div \rho \}+\frac{\beta}{m}(n-1)(E_{1} \omega) \nonumber\\&&
+\frac{1}{m}\{\alpha_{1} E_{1} (\omega)+\beta_{1}\rho (\omega)A(E_{1})-n \alpha_{1} E_{1} (\omega)+\beta_{1} E_{1} (\omega)\}\nonumber\\&&
+(1-n)E_{1} (\beta).
\end{eqnarray}
Setting $E_{1}=\rho$ in the last equation and (\ref{1}), and then comparing both the equations, we find
\begin{equation}
\label{k9}
\{\frac{m}{1-n}(\alpha_{1}-\beta_{1})+\beta-\alpha_{1}\} \rho (\omega)=m \{(1-n)[\rho (\alpha_{1})+\rho(\beta)- \beta_{1} div \rho \}.
\end{equation}
Let $\alpha_{1}$, $\beta$ and $\omega$ are invariant under the velocity vector field $\rho$ which mean $\rho (\alpha_{1})=0$, $\rho (\beta)=0$ and $\rho (\omega)=0$. Hence, the foregoing equation yields either, $div \rho =0$, or $\beta_{1} =0$, since $m\neq 0$.\par
If $div \rho =0$, then the velocity vector field is conservative. Since a conservative vector field is always irrotational, we get the vorticity of the perfect fluid is zero.\par
If $\beta_{1} =0$, then we obtain that $\sigma +p=0$. This represents dark energy era.\par
This finishes the proof.

\section{{\bf Example}}
Let $\mathcal{M}^3=\{(x,y,z) \in \mathcal{R}^3 : z>0 \}$, where $(x,y,z)$ are the standard co-ordinate of $\mathcal{R}^3$. We choose the Riemannian metric $g$ on $\mathcal{M}^3$ described by
$$g=dx \otimes dx+dy \otimes dy+\frac{x^2+y^2+e^{-2z}}{4z}dz \otimes dz-\frac{y}{\sqrt{z}}dx \otimes dz-\frac{x}{\sqrt{z}}dy \otimes dz.$$
Let $$\delta_1=\frac{\partial}{\partial x},\quad \delta_2=\frac{\partial}{\partial y}, \quad \delta_3=ye^z\frac{\partial}{\partial x}+xe^z\frac{\partial}{\partial y}+2\sqrt{z}e^z\frac{\partial}{\partial z}.$$
Then, the vector fields $\delta_1,\delta_2,\delta_3$ are orthonormal.\par

We have
    $$[\delta_1,\delta_2]=0, \quad [\delta_1,\delta_3]=e^z \delta_2, \quad [\delta_2,\delta_3]=e^z \delta_1.$$
We define $\eta$, $\zeta$ and $\varphi$ by
    $$\eta=\frac{e^{-z}}{2\sqrt{z}}\,dz,\quad \zeta=\delta_3,\quad \varphi \delta_1=\delta_2,\quad \varphi \delta_2=-\delta_1,\quad \varphi \delta_3=0.$$

 We can easily verify that, $\eta$ and $\Phi$ are closed. Hence, $\mathcal{M}^3$ is an almost co-K\"ahler manifold.\par

Using Koszul's formula we get
$$\nabla_{\delta_1}\delta_1=0, \quad \nabla_{\delta_1}\delta_2=-e^z \delta_3, \quad \nabla_{\delta_1}\delta_3=e^z \delta_2,$$
$$\nabla_{\delta_2}\delta_1=-e^z\delta_3, \quad \nabla_{\delta_2}\delta_2=0, \quad \nabla_{\delta_2}\delta_3=e^z\delta_1,$$
    $$\nabla_{\delta_3}\delta_1=0, \quad \nabla_{\delta_3}\delta_2=0, \quad \nabla_{\delta_3}\delta_3=0.$$
    The components of the curvature tensor $K$ are given by
    $$K(\delta_1,\delta_2)\delta_1=-e^{2z}\delta_2, \quad K(\delta_1,\delta_2)\delta_2=e^{2z}\delta_1, \quad K(\delta_1,\delta_2)\delta_3=0,$$
    $$K(\delta_1,\delta_3)\delta_1=e^{2z}\delta_3, \quad K(\delta_1,\delta_3)\delta_2=2\sqrt{z}e^{2z}\delta_3,$$
     $$K(\delta_1,\delta_3)\delta_3=-e^{2z}\delta_1-2\sqrt{z}e^{2z}\delta_2, \quad K(\delta_2,\delta_3)\delta_1=2\sqrt{z}e^{2z}\delta_3,$$
    $$K(\delta_2,\delta_3)\delta_2=e^{2z}\delta_3, \quad K(\delta_2,\delta_3)\delta_3=-2\sqrt{z}e^{2z}\delta_1-e^{2z}\delta_2.$$

Utilizing the foregoing expressions of $K$, we compute the Ricci operator $Q$ by
    \be Q\delta_1=-2\sqrt{z}e^{2z}\delta_2, \quad Q\delta_2=-2\sqrt{z}e^{2z}\delta_1, \quad Q\delta_3=-2e^{2z}\delta_3. \label{501}\ee
    The tensor field $h$ is given by
    \be h\delta_1=-e^z \delta_1,\quad h\delta_2=e^z \delta_2, \quad h\delta_3=0. \ee
    Suppose that $\omega=z$. Then, $D\omega=2\sqrt{z}e^z \delta_3$. By directed computation, we have
    \be
    \begin{cases}
      \nabla_{\delta_1}D\omega=2\sqrt{z}e^{2z}\delta_2,\\
    	\nabla_{\delta_2}D\omega=2\sqrt{z}e^{2z}\delta_1,\\
    	\nabla_{\delta_3}D\omega=2e^{2z}(1+2z)\delta_3.
    	\end{cases}
    \label{502}
    	\ee
   From (\ref{501}) and (\ref{502}), we can verify that
   $$
   \begin{cases}
   Q\delta_1+\nabla_{\delta_1}D\omega-g(\delta_1,D\omega)D\omega=0,\\
   	Q\delta_2+\nabla_{\delta_2}D\omega-g(\delta_2,D\omega)D\omega=0,\\
   	Q\delta_3+\nabla_{\delta_3}D\omega-g(\delta_3,D\omega)D\omega=0.
   \end{cases}
   $$
   Thus for $m=1,\, \rho=\lambda_{1}=0$, equation (\ref{104}) is satisfied. Hence $M$ is an $(m,\rho)$-quasi-Einstein manifold.
\section{Conclusion}
Einstein solitons are essential in both mathematics and physics. It is intriguing to look into Einstein solitons and their generalizations in Riemannian and Semi-Riemannian geometry. Recent years have seen the development of many generalizations of Einstein solitons, including quasi-Einstein solitons, generalized quasi-Einstein solitons, $m$-quasi-Einstein solitons, $(m,\rho)$-quasi-Einstein solitons. In this study, we investigate the $(m,\rho)$-quasi-Einstein solitons in almost co-K\"ahler manifolds.\par

Here we show that an almost co-K\"ahler manifold admitting $(m,\rho)$-quasi-Einstein solitons is locally isomorphic to a solvable non-nilpotent Lie group. We also establish that there does not exist $(m,\rho)$-quasi-Einstein structure on a compact $(\kappa,\mu)$-almost co-K\"ahler manifold. Finally, we construct an example of an almost co-K\"ahler manifold with $(m,\rho)$-quasi-Einstein solitons.\par

Further study of $(m,\rho)$-quasi-Einstein solitons in the context of Hopf manifolds, as well as in the general theory of relativity and cosmology, is possible in the near future.

\section{Acknowledgement}
We would like to thank the Referees and the Editor for reviewing the paper carefully and their valuable comments to improve the quality of the paper.


\begin{thebibliography}{99}
		
\bibitem{11} Basse, A. L., {\it Einstein manifolds}, Springer, Berlin, 1987.		
		
\bibitem{23} Barros, A. and Gomes, J. N., {\it Triviality of compact $m$-quasi-Einstein manifolds}, Res. Math., {\bf 71} (2017), 241-250.
		
\bibitem{10} Blair, D. E., {\it Riemannian geometry of contact and symplectic manifolds}, {\bf 203}, Birkh\"auser, 2010.
				
\bibitem{5} Blair, D. E., {\it The theory of quasi-Sasakian structures}, J. Diff. Geom., {\bf 1} (1967), 331-345.
		
\bibitem{bbw} Boothby, M.M. and R. C. Wong, R.C., {\it On contact manifolds}, Ann. Math. {\bf68} (1958), 421–450.
\bibitem{cal} R. R. Caldwell, M. Kamionkowski, N. N. Weinberg, {\it Dark Energy with $w<-1$ causes a
cosmic doomsday}, Phys. Rev. Lett. 91 (2003) 071301.
\bibitem{cal1} R. R. Caldwell, {\it A Phantom Menace? Cosmological consequences of a dark energy component
with super-negative equation of state}, Phys. Lett. B 545 (2002) 23–29.
\bibitem{17} Cao, H. D., {\it Geometry of Ricci solitons}, Chinese Ann. Math. Soc., {\bf 27} (2006), 121-142.
		
\bibitem{12} Catino, G., {\it Generalized quasi-Einstein manifolds with harmonic Weyl tensor}, Math. Z., {\bf 271} (2012), 751-756.
		
\bibitem{4} Cappelletti-Montano, B., Nicola, A. D. and Yudin, I., {\it A survey on cosymplectic geometry}, Rev. Math. Phys., {\bf 25} (2013), 1343002, 55pp.
\bibitem{ch1}P.H. Chavanis, {\it Cosmology with a stiff matter era,} Phys. Rev. D {\bf 92}, 103004 (2015).
\bibitem{chen} Chen, X. M., {\it Cotton solitons on almost co-Kahler 3-manifolds}, Quaestiones Mathematicae {\bf44} (2021), 1055–1075.
		
\bibitem{14} Dacko, P., {\it On almost cosymplectic manifolds with the structure vector field $\zeta$ belonging to the $\kappa$-nullity distribution}, Balkan J. Geom. Appl., {\bf 5} (2000), 47-60.
		
\bibitem{22} Dacko, P. and Olszak, Z., {\it On almost cosymplectic $(\kappa,\mu,\nu)$-spaces}, Banach center publ., {\bf 69} (2005), 211-220.
\bibitem{dede} De, K and De, U.C., {\it $(m,\rho)$-quasi Einstein solitons on paracontact geometry}, Novi Sad J. Math. (published online)
    https://doi.org/10.30755/NSJOM.13225.
\bibitem{dez}De, K., De, U.C., Syied, A. A., Turki N. B. and Alsaeed, S., {\it Perfect fluid spacetimes and gradient solitons}, Journal of Nonlinear Mathematical Physics
https://doi.org/10.1007/s44198-022-00066-5.
\bibitem{dms} De, U. C., Mantica, C. A. and Suh, Y. J., {\it Perfect Fluid Spacetimes and Gradient Solitons}, Filomat, {\bf{36}} (2022), 829-842.
\bibitem{de} De, U.C. Chaubey, S.K. and Shenawy, S., {\it Perfect fluid spacetimes and Yamabe solitons}, J Math Phys. {\bf 62}, 032501 (2021); https://doi.org/10.1063/5.0033967		
		
\bibitem{25} De, U.C. and De, B.K., {\it On quasi-Einstein manifolds}, Commun. Korean Math. Soc., {\bf 23} (2008), 413-420.
		
\bibitem{20} Demirbe$\breve{{\rm g}}$, S. A. and G\"uler, S., {\it Rigidity of $(m,\rho)$-quasi Einstein manifolds}, Mathematische Nachrichten, {\bf 290} (2017), 2100-2110.
\bibitem{dug} Duggal, K.L., {\it A New Class of Almost Ricci Solitons and Their Physical Interpretation,} Int Sch Res Notices, (2016), 4903520 (6 pages). doi: 10.1155/2016/4903520.
\bibitem{dug1} Duggal, K.L., {\it Almost Ricci Solitons and Physical Applications,} IEJG, {\bf 10} (2017), 1–10.
		
\bibitem{1} Endo, H., {\it Non-existence of almost cosymplectic manifolds satisfying a certain condition}, Tensor (N.S.), {\bf 63} (2002), 272-284.
		
				
\bibitem{3} Goldberg, S.I. and Yano. K., {\it Integrability of almost cosymplectic structures}, Pacific J. Math, {\bf 31} (1969), 373-382.
				
\bibitem{13} Huang, G. and Wei, Y., {\it The classification of $(m,\rho)$-quasi-Einstein manifolds}, Ann. Global Anal. Geom., {\bf 44} (2013), 269-282.
\bibitem{fpu}A. Mussot, A. Kudlinski, M. Droques, P. Szriftgiser and N. Akhmediev, {\it Fermi-Pasta-Ulam recurrence in Nonlinear Fiber Optics: The Role of Reversible and Irreversible Losses}, Physical Review X 4, 011054 (2014). DOI:https://doi.org/10.1103/PhysRevX.4.011054
				
\bibitem{6} Olszak, Z., {\it On almost cosymplectic manifolds}, Kodai Math. J., {\bf 4} (1981), 239-250.
				
\bibitem{7} Olszak, Z., {\it On almost cosymplectic manifolds with K\"ahlerian leaves,} Tensor (N.S.), {\bf 46} (1987), 117-124.
				
\bibitem{8} Perrone, D., {\it Classification of homogeneous almost cosymplectic three-manifolds}, Diff. Geom. Appl., {\bf 30} (2012), 49-58.
				
\bibitem{9} Perrone, D., {\it Minimal Reeb vector fields on almost cosymplectic manifolds}, Kodai Math. J., {\bf 36} (2013), 258-274.
		
\bibitem{sas} Sasaki, S., {\it On differentiable manifolds with certain structures which are closely related to almost contact structure I}, Tohoku Math. J. {\bf12} (1960), 459–476.
		
\bibitem{18} Suh, Y. J. and De, U. C., {\it Yamabe solitons and Ricci solitons on  almost co-K\"ahler manifolds}, Canadian Math. Bull., {\bf 62} (2019), 653-661.
		
\bibitem{tano}Tanno, S., {\it  Note on infinitesimal transformations over contact manifolds}, Tohoku Math.J. {\bf14} (1962), 416–430.
		
\bibitem{30} Wang, W., {\it A class of three dimensional  almost co-K\"ahler manifold}, Palestine J. Math. {\bf 6} (2017), 111-118.
			
\bibitem{2} Wang, Y., {\it A generalization of Goldberg conjecture for co-K\"ahler manifolds}, Mediterr. J. Math., {\bf 13} (2016), 2679-2690.
			
				
				
\bibitem{21} Wang, Y., {\it Ricci tensors on three-dimensional  almost co-K\"ahler manifolds}, Kodai Math. J., {\bf 39} (2016), 469-483.
\bibitem{ze} Y. B. Zeldovich, {\it The equation of state of ultrahigh density and its relativistic limitations},
Soviet Phys. JETP {\bf14} n.5 (1962), 1143-1147.	    		
		
		
\end{thebibliography}
\end{document}